\documentclass[11pt]{article}
\title{On Laplacian and Signless Laplacian Permanental Polynomials of Some Well-known Graphs}

\author{Sarbari Mitra, Soumya Bhoumik\\
Department of Mathematics\\
Fort Hays State University
}
\usepackage{amssymb}
\usepackage[all]{xy}
\usepackage{amsmath}
\usepackage{algorithm}
\usepackage[noend]{algpseudocode}
\usepackage{amssymb}
\usepackage{array}
\usepackage{theorem}
\usepackage{graphicx}
\usepackage{cite}
\usepackage{colortbl}
\newtheorem{thrm}{Theorem}[section]

\newtheorem{lem}[thrm]{Lemma}
\newtheorem{cor}[thrm]{Corollary}

\usepackage{amsfonts}
\usepackage{graphicx}
\usepackage{graphics}
\usepackage{amsfonts}
\usepackage{multirow}
\usepackage{caption}
\usepackage{subcaption}
\usepackage{xy}
\usepackage{graphicx}
\usepackage{xmpmulti}
\newtheorem{hey}[thrm]{Remark}
\newtheorem{conj}[thrm]{Conjecture}

 \newcounter{case}

 \renewcommand{\thecase}{\arabic{case}}

\newcounter{subcase}

 \renewcommand{\thesubcase}{\alph{subcase}}

\usepackage{fullpage}

\def\CT{{\rm CT}}

\newcounter{cases}
\newcounter{subcases}

\newenvironment{proof}{\noindent {\sc Proof}.}
                {\phantom{a} \hfill \framebox[2.2mm]{ } \bigskip}

\makeatletter
\def\BState{\State\hskip-\ALG@thistlm}
\makeatother

\providecommand{\keywords}[1]{\textbf{\textit{Keywords.}} #1}
\providecommand{\subjclass}[1]{\textbf{\textit{Mathematics Subject Classification.}} #1}

\begin{document}

\pagestyle{plain}

\baselineskip = 1.2\normalbaselineskip
%\subjclass{05C78}

\maketitle

\begin{abstract}

The permanent of an $n \times n$ matrix $M = (m_{ij})$ is defined as $\mathrm{per}(M) = \sum_{\sigma \in S_n} \prod_{i=1}^n m_{i,\sigma(i)}$, where $S_n$ denotes the symmetric group on $\{1,2,\ldots,n\}$. The permanental polynomial of $M$, is defined by $\psi(M;x) = \mathrm{per}(xI_n - M)$. We study two fundamental variants: the Laplacian permanental polynomial $\psi(L(G);x)$ and signless Laplacian permanental polynomial $\psi(Q(G);x)$ of a graph $G$. A graph is said to be {determined} by its (signless) Laplacian permanental polynomial if no other non-isomorphic graph shares the same polynomial. A graph is combinedly determined when isomorphism is guaranteed by the equality of both polynomials. Characterizing which graphs are determined by their(signless) Laplacian permanental polynomials is an interesting problem. This paper investigates the permanental characterization problem for several families of starlike graphs, including: spider graphs (tree), coconut tree, perfect binary tree, corona product of $C_m$ and $K_n$, and $\bar K_n$ for various values of $m$ and $n$. We establish which of these graphs are determined by their Laplacian or signless Laplacian permanental polynomials, and which require both polynomials for complete characterization. We emphasize that in this manuscript, we have considered a few techniques to compute the permanental polynomial of matrices and their propagation.

% \textcolor{red}{In general, if we consider a graph of degree $n$, then the number of terms in the permanental polynomial will be $n!$. Computing the permanental polynomial of a Matrix of sufficiently higher order (even with 0-1 entries) is a $\#P$ problem. We emphasize that in this manuscript, we have considered a few techniques to compute the permanental polynomial of matrices and their propagation. }
\end{abstract}



\section{Introduction}
All graphs considered throughout the paper are simple and undirected. Let \( G = (V(G), E(G)) \) be a graph with vertex set \( V(G) \) and edge set \( E(G) \) with respective cardinalities $n,m$. The adjacency matrix of $G$, denoted by $A(G )=(a_{ij})_{n\times n}$, is an $n\times n$ symmetric matrix such that $a_{ij}=1$ if vertices $v_i$ and $v_j$ are adjacent and $0$ otherwise. The degree matrix of \( G \), denoted by \( D(G) \), is defined as \( D(G) = \text{diag}(d_1, d_2, \dots, d_n) \), where \( d_i \) represents the degree of vertex \( v_i \). The Laplacian matrix and signless Laplacian matrix of \( G \) are given by $L(G) = D(G)-A(G)$ and $Q(G) = D(G) + A(G)$, respectively. For an \( n \times n \) matrix \( M \) with entries \( m_{ij} \) (\( i, j = 1, 2, \dots, n \)), the permanent of \( M \) is defined as \( \text{per}(M) = \sum_{\sigma\in S_n} \prod_{i=1}^{n} m_{i\sigma(i)}\) (the sum ranges over all permutations \( \sigma \) of \( \{1, 2, \dots, n\} \)). The permanental polynomial of \( M \), denoted by \( \psi(M; x) = \text{per}(xI_n - M)\), where \( I_n \) is the \( n \times n \) identity matrix. Specifically, the polynomials \( \psi(L(G); x) \) and \( \psi(Q(G); x) \) are referred to as the Laplacian permanental polynomial and signless Laplacian permanental polynomial of \( G \), respectively. A graph \( G \) is said to be determined by its Laplacian permanental polynomial (or signless Laplacian permanental polynomial) if every graph with the same Laplacian permanental polynomial (respectively, signless Laplacian permanental polynomial) as \( G \) is isomorphic to \( G \). A graph $G$ is {combinedly determined} by its Laplacian and signless Laplacian permanental polynomials if any graph $H$ satisfying $\psi(L(H);x) = \psi(L(G);x)$ and $\psi(Q(H);x) = \psi(Q(G);x)$ must be isomorphic to $G$.

% {\color{blue} 

% The concept of permanent, introduced by Cauchy \cite{smithies2005cauchy} in 1812, is a variant of the determinant where the signs of permutations are ignored, and both are special cases of a broader matrix function called the immanant, defined by Littlewood and Richardson \cite{littlewood1934group}.} 
Valiant \cite{valiant1979complexity} first proved that computing the permanent matrix, even for the $(0,1)$ matrices, is a \#P-complete problem. Turner \cite{turner1968generalized} later introduced the concept of adjacency permanental polynomial of graphs, followed by Merris \cite{merris1981permanental}, Kasum \cite{kasum1981chemical}, Borowiecki \cite{borowiecki1982computing}, and others. The study of Laplacian permanental polynomials of graphs was initiated by Merris et al. \cite{merris1981permanental}, leading to two main research directions. The first focuses on computing the coefficients of Laplacian permanental polynomials for various graphs \cite{merris1982laplacian, brualdi1984permanent,goldwasser1986permanent, bapat1986bound, cash2004lapacian,geng2010further,geng2015permanental}. Signless Laplacian permanental polynomials are, however, comparatively less explored. Faria \cite{faria1985permanental} first considered the signless Laplacian permanental polynomial of $G$, and showed that $\psi(L(G);x) = \psi(Q(G);x)$ when $G$ is a bipartite graph. Furthermore, he discussed the multiplicity of integer roots of $\psi(Q(G);x)$ \cite{faria1995multiplicity}. See \cite{faria1985permanental,li2011permanental,li2012permanental} for more results on the permanental polynomial of the signless Laplacian matrix. In 2017, Liu et al. \cite{liu2017computing} developed recursive methods to compute both Laplacian and signless Laplacian permanental polynomials. 

The second direction focused on characterizing graphs through their (signless) Laplacian permanental polynomials, an approach first investigated by Merris et al. \cite{merris1981permanental}. Liu \cite{liu2019signless} proved that graphs with up to 6 vertices (respectively, 7 vertices) are uniquely determined by their Laplacian (signless Laplacian) permanental polynomials, and linked the polynomial’s constant term to graph connectedness. Later work \cite{liu2022graphs,wu2022further} extended this to paths, cycles, lollipop graphs, and other families via the signless Laplacian permanental polynomial. Khan et al. \cite{khan2023study} showed that star, wheel, friendship, and certain caterpillar graphs are uniquely identified by their Laplacian and signless Laplacian permanental polynomials.

{ In this work, we investigate several families of starlike graphs, beginning with the coconut tree $\CT_{m,n}$, formed by attaching $m$ paths of length $n$ to a central vertex, and the regular spider graph $\mathcal{S}_{n,m}$, which consists of $n$ legs each of length $m$ emanating from a central vertex. We also examine perfect binary trees $T_k$, which are complete binary trees with $2k$ leaves at depth $k$, as well as corona graphs: $C_n \odot K_1$ for $n \geq 3$, obtained by attaching a pendant vertex to each vertex of a cycle $C_n$, and $C_m \odot \bar{K}_n$ for $m = 3, 4$. Our findings demonstrate that these families are uniquely characterized by their (signless) Laplacian permanental polynomials. Specifically, we establish that the spider graphs $\CT_{m,n}$ and $\mathcal{S}_{n,m}$, along with all perfect binary trees, are determined by both their Laplacian and signless Laplacian permanental polynomials. Furthermore, we prove that this characterization also holds for smaller instances, including the perfect binary trees $T_2$ and $T_3$, the cycle coronas $C_n \odot K_1$ with $n \leq 5$, and the graphs $C_n \odot \bar{K}_m$ where $n = 3, 4$.}

\section{Preliminaries}

First, we need to establish some notation and results for working with principal submatrices of graph Laplacians. Given any graph $G$ and a subset of vertices $S \subseteq V(G)$, we will denote the principal submatrix of the Laplacian $L(G)$ by $L_S(G)$, obtained by deleting the rows and columns corresponding to each vertex in $S$. The same notation applies to the signless Laplacian $Q(G)$, where we write $Q_S(G)$ for the corresponding principal submatrix. In particular when $S$ contains just a single vertex $v$, we simplify the notation to $L_v(G)$ (analogously $Q_v(G)$). Similarly, when dealing with an edge $e = \{u,v\} \in E(G)$, we use $L_e(G)$ (respectively $Q_e(G)$) to denote the principal submatrices obtained by removing both endpoints of $e$. The following lemma (\cite{liu2017computing}, Theorems 1.2 and 1.3) establishes recursive formulas for calculating both the Laplacian and signless Laplacian permanental polynomials of a graph.

\begin{lem}\label{reductionlemma}
\begin{enumerate}
    \item[(i)] Let $v$ be a vertex of a graph $G$, $\mathcal{C}_G(v)$ be the set of cycles of $G$ containing $v$, and $N(v)$ be the set of vertices of $G$ adjacent to $v$. Then
    \begin{align*}
    \psi(L(G);x) &= (x - d(v)) \psi(L_v(G)) + \sum_{u \in N(v)} \psi(L_{vu}(G)) 
    + 2 \sum_{C \in \mathcal{C}_G(v)} \psi(L_{V(C)}(G)), \\
    \psi(Q(G);x) &= (x - d(v)) \psi(Q_v(G)) + \sum_{u \in N(v)} \psi(Q_{vu}(G)) 
    + 2 \sum_{C \in \mathcal{C}_G(v)} (-1)^{|V(C)|} \psi(Q_{V(C)}(G)).
    \end{align*}

    \item[(ii)] Let $e = \{u,v\}$ be an edge of a graph $G$ and $G - e$ be the graph obtained by deleting the edge $e$ from $G$. Then
    \begin{align*}
    \psi(L(G);x) &= \psi(L(G - e)) - \psi(L_v(G - e)) - \psi(L_u(G - e)) 
    + 2 \psi(L_e(G)) \\
    & \quad + 2 \sum_{C \in \mathcal{C}_G(e)} \psi(L_{V(C)}(G)), \\
    \psi(Q(G);x) &= \psi(Q(G - e)) - \psi(Q_v(G - e)) - \psi(Q_u(G - e)) 
    + 2 \psi(Q_e(G)) \\
    & \quad + 2 \sum_{C \in \mathcal{C}_G(e)} (-1)^{|V(C)|} \psi(Q_{V(C)}(G)).
    \end{align*}
\end{enumerate}
\end{lem}

The next result demonstrates how key graph parameters can be recovered from the Laplacian and signless Laplacian permanental polynomials.

\begin{lem}\cite{liu2022graphs}\label{firstthreevalues}
    Let $G$ be a graph with $n$ vertices and $m$ edges, and let $(d_1, d_2, \dots, d_n)$ be its degree sequence. 
Suppose that 
\[
\psi(L(G); x) = \sum_j (-1)^j l_j(G) x^{n-j} 
\quad \text{and} \quad
\psi(Q(G); x) = \sum_j (-1)^j q_j(G) x^{n-j}.
\]
Then
\[
l_0(G) = q_0(G) = 1,
\]
\[
l_1(G) = q_1(G) = 2m,
\]
\[
l_2(G) = q_2(G) = 2m^2 + m - \frac{1}{2} \sum_{i=1}^n d_i^2,
\]
\[
l_3(G) = \frac{1}{3} \left(
-6 t_G + 6 m^2 - 3 \sum_{r=1}^n d_r^2 + 4 m^3 
- 3 m \sum_{r=1}^n d_r^2 + \sum_{r=1}^n d_r^3
\right),
\]
\[
q_3(G) = \frac{1}{3} \left(
6 t_G + 6 m^2 - 3 \sum_{r=1}^n d_r^2 + 4 m^3 
- 3 m \sum_{r=1}^n d_r^2 + \sum_{r=1}^n d_r^3
\right),
\]
% \begin{align*}
% l_4(G) = & -\frac{1}{4} \sum_{r=1}^n d_r^4 +\bigg(\frac{2}{3}m+1 \bigg)\sum_{r=1}^n d_r^3 - \frac{1}{2}(2m^2+5m+1)\sum_{r=1}^n d_r^2 +\frac{1}{8} \bigg(\sum_{r=1}^n d_r^2\bigg)^2 +\sum_{v_iv_j \in E(G)}d_id_j \\&  +2\sum_{r=1}^nd_r*t_{G_{v_r}}+2s_G -4mt_G+\frac{2}{3}m^4 +2m^3 +\frac{1}{2}m^2 +\frac{1}{2}m,
% \end{align*}

where $t_G$ denotes the number of triangles $(C_3)$ in $G$.
\end{lem}
The following results are immediate consequences of the Lemma \ref{firstthreevalues}.

\begin{cor}\cite{liu2022graphs}\label{corolarydetermination}
For any graph $G$ with $t(G)$ triangles and $d_i$ being the degree of the vertex $v_i$ of $G$.
\begin{enumerate}
    \item[(i)] Both $\psi(L(G);x)$ and $\psi(Q(G);x)$ determine $\vert V(G)\vert, \vert E(G)\vert$, and $\sum_i d_i^2$ (sum of the square of the degrees of all vertices).
    \item[(ii)] Additionally $\psi(L(G);x)$ determines $-6t(G) + \sum_i d_i^3$, whereas $\psi(Q(G);x)$ determines $6t(G) + \sum_i d_i^3$.
\end{enumerate}
\end{cor}

The following result first appeared in \cite{faria1985permanental}, though an alternative proof was later presented in \cite{liu2017computing}.

\begin{thrm}
For any bipartite graph $G$, the Laplacian permanental polynomial $\psi(L(G);x)$ coincides with its signless Laplacian permanental polynomial $\psi(Q(G);x)$.
\end{thrm}

\section{Spider Tree}
This section examines two specific types of spider trees: the coconut tree, analyzed in Section \ref{coco}, and the regular spider tree, discussed in Section \ref{spider}.

\subsection{coconut Tree Graphs} \label{coco}
A coconut tree $\CT_{m,n}$ is the graph obtained from the path $P_m$ by appending $n$ pendant edges at an end vertex of $P_m$, for all positive integers $m,n\ge 2$. We identify the vertices of a coconut tree graph as $V(G)= \{u_1,u_2,\cdots,u_m\}\cup \{v_1,v_2,\cdots,v_n\}$, where $u_1,u_2,\cdots,u_m$ are the vertices of path $P_m$ and $v_1,v_2,\cdots,v_n$ are the pendant vertices adjacent to $u_1$. We note that $\CT_{m,n}$ for $m<3$ is simply the star graph $S_n$, which has been covered in the paper by \cite{khan2023study, liu2019signless}. We follow the notation of \cite{khan2023study} where $d_{\max}$ denotes the maximum degree of vertices in $G$, and $k_i$ denotes the number of vertices of degree $i$ in $G$, $i= 0,1,\cdots, d_{\max}$.

\begin{thrm}
For the coconut Tree graph $\CT_{3,n}$,
\begin{multline*}
\psi(L(\CT_{3,n}); x) = \psi(Q(\CT_{3,n}); x) = x^{n+3} - (2n+4) x^{n+2} + \left( \frac{3n^2 +15n + 14}{2} \right) x^{n+1} \\ - \frac{2n^3 +15n^2 + 31n + 12)}{3} x^{n} + \cdots + 2(3n+2)(-1)^n
\end{multline*}
\end{thrm}
\begin{proof}
Since $\CT_{3,n}$ is a bipartite graph, $\psi(L(\CT_{3,n});x)=\psi(Q(\CT_{3,n});x)$. Applying Lemma \ref{reductionlemma} to $\psi(L(\CT_{3,n});x)$, with $u_1$ as the central vertex, we obtain: 
\begin{equation}\label{CT_{3,n}}
\psi(L(\CT_{3,n}); x) = (x-n-1)\psi(L_{u_1}(\CT_{3,n}); x)+\psi(L_{u_1u_2}(\CT_{3,n}); x)+n\psi(L_{u_1v_1}(\CT_{3,n}); x)
\end{equation} 
We have the following polynomials of the submatrices:
\begin{align*}
\psi(L_{u_1}(\CT_{3,n}); x) &= (x-2)(x-1)^{n+1} + (x-1)^n, \\
\psi(L_{u_1u_2}(\CT_{3,n}); x) &= (x-1)^{n+1}, \\
\psi(L_{u_1v_1}(\CT_{3,n}); x) &= (x-2)(x-1)^n + (x-1)^{n-1}.
\end{align*}

\noindent Substituting these into Eqn.(\ref{CT_{3,n}}) and simplifying yields the desired result.
\end{proof}

\begin{thrm}\label{CT{3,n}}
 $\CT_{3,n}$ is determined by its (signless) Laplacian permanental polynomial.     
\end{thrm}
\begin{proof}
From the above theorem, using the Lemma \ref{firstthreevalues} and Corollary \ref{corolarydetermination}, we have 
\begin{align*}
    \sum_{i=0}^{d_{\max}} k_i = n+3,  \sum_{i=0}^{d_{\max}} i k_i = 2n+4,  \sum_{i=0}^{d_{\max}} i^2 k_i = n^2 + 3n + 6,  -6t_G+ \sum_{i=0}^{d_{\max}} i^3 k_i = n^3 + 3n^2 + 4n + 10. \label{3eq:sum4}
\end{align*}
A key observation follows from these equations: 
\begin{equation}
    6t_G - \sum_{i=0}^{d_{\max}} (i-1)(i-2)(i-n-1)k_i = 0. \label{3eq:key}
\end{equation}
Since each term on the left side is non-negative in Equation \ref{3eq:key}, we conclude $t_G = 0$ and consequently $ k_3 = k_4 = \cdots = k_n = 0$.  Thus, the only non-zero values are: $k_1=n+1, k_2=1$, and $k_{n+1}=1$. Since $G$ has a total of $n+3$ vertices and one vertex $v$ of degree $n+1$, $v$ connects all vertices except one. Now, as there is only one vertex of degree $2$, it is isomorphic to the coconut tree $\CT_{3,n}$.
\end{proof}

%%%%%%%coconut Tree graph $\CT_{4,n}$

\begin{thrm}\label{CT{4,n}}
For the coconut Tree graph $\CT_{4,n}$,
\begin{multline*}
\psi(L(\CT_{4,n}); x) = \psi(Q(\CT_{4,n}); x) = x^{n+4} - (2n+6) x^{n+3} + \left( \frac{3n^2 +23n + 32}{2} \right) x^{n+1} \\ - \frac{2n^3 +24n^2 + 82n + 60)}{3} x^{n} + \cdots + (14n+10)(-1)^n
\end{multline*}
\end{thrm}
\begin{proof}
The equality \( \psi(L(\CT_{4,n}); x) = \psi(Q(\CT_{4,n}); x) \) follows from the bipartiteness of \( \CT_{4,n} \). Letting $u_1$ be the central vertex, applying Lemma \ref{reductionlemma} to $\psi(L(\CT_{4,n});x)$, we obtain: 
\begin{equation}\label{CT_{4,n}}
\psi(L(\CT_{4,n}); x) = (x-n-1)\psi(L_{u_1}(\CT_{4,n}); x)+\psi(L_{u_1u_2}(\CT_{4,n}); x)+n\psi(L_{u_1v_1}(\CT_{4,n}); x)
\end{equation} 
The constituent polynomials are computed as follows:
\begin{align*}
\psi(L_{u_1}(\CT_{4,n}); x) &= (x-2)^2(x-1)^{n+1} + (x-1)^{n+1}+ (x-2) (x-1)^n, \\
\psi(L_{u_1u_2}(\CT_{4,n}); x) &= (x-2)(x-1)^{n+1} +(x-1)^{n}, \\
\psi(L_{u_1v_1}(\CT_{4,n}); x) &= (x-2)^2(x-1)^n+ (x-2)(x-1)^{n-1} + (x-1)^{n}.
\end{align*}
\noindent Substituting these into Equation \ref{CT_{4,n}} and simplifying yields the desired result.
\end{proof}

\begin{thrm}
 $\CT_{4,n}$ is determined by its (signless) Laplacian permanental polynomial.     
\end{thrm}

\begin{proof}
Similar to Theorem \ref{CT{3,n}} we now have 
\begin{align*}
    \sum_{i=0}^{d_{\max}} k_i = n+4, \sum_{i=0}^{d_{\max}} i k_i = 2n+6,
    \sum_{i=0}^{d_{\max}} i^2 k_i = n^2 + 3n +10, && 
    -6t_G+ \sum_{i=0}^{d_{\max}} i^3 k_i = n^3 + 3n^2 + 4n + 18. 
\end{align*}
On solving the system of equations, we get, 
\begin{equation}
    6t_G - \sum_{i=0}^{d_{\max}} (i-1)(i-2)(i-n-1)k_i = 0. \label{4eq:key1}
\end{equation}
Since each term on the left side is non-negative in Equation \ref{4eq:key1}, we conclude $t_G = k_3 = k_4 = \cdots = k_n = 0$. This leaves only three non-zero vertex counts: $k_1=n+2, k_2=1$, and $k_{n+1}=1$ implying that graph $G$ has $n+4$ vertices in total. The structure of $G$ consists of a central vertex $v$ of degree $n+1$ (adjacent to all but two vertices), exactly one vertex of degree $2$, and finally $n+1$ leaves (vertices of degree 1). There are two possible non-isomorphic graphs of this degree distribution: one is $\CT_{4,n}$, and the other is $H_{3,n}$ (see Figure \ref{Hgraph}). However, applying Lemma \ref{reductionlemma} to the central vertex $v_1$ shows that the Laplacian permanental polynomial of $H_{3,n}$ differs from that in Theorem \ref{CT{4,n}}. We omit the detailed computation as it follows similar but lengthy steps.
\begin{multline*}
\psi(L(H_{3,n}); x) = \psi(Q(H_{3,n}); x) = x^{n+4} - (2n+6) x^{n+3} + \left( \frac{3n^2 +23n + 32}{2} \right) x^{n+1} \\ - \frac{2n^3 +24n^2 + 82n + 60}{3} x^{n} + \cdots + (18n-15)(-1)^n
\end{multline*}
\end{proof}

% {\color{blue}

% \begin{thrm}
% For the coconut Tree graph $\CT_{5,n}$,
% \begin{multline*}
% \psi(L(\CT_{5,n}); x) = \psi(Q(\CT_{5,n}); x) = x^{n+4} - (2n+6) x^{n+3} + \left( \frac{3n^2 +23n + 32}{2} \right) x^{n+1} \\ - \frac{2n^3 +24n^2 + 82n + 60)}{3} x^{n} + \cdots + (14n+10)(-1)^n
% \end{multline*}
% \end{thrm}
% }
% \begin{proof}
% Note that $\psi(L_{v_1}(\C))$
% \end{proof}

\begin{figure}
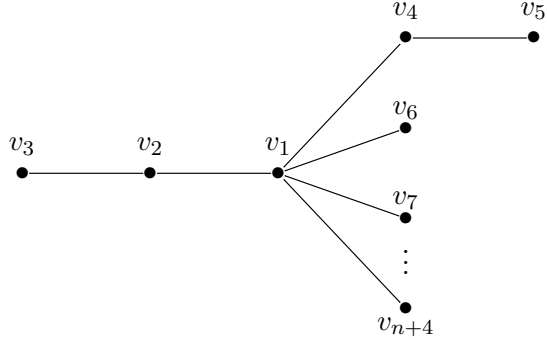

\[ \xygraph{
!{<0cm,0cm>;<1.7 cm,0cm>:<0cm,1.2cm>::}
!{(-0.5,0) }*{\bullet}="2" !{(-0.5,0.3) }*{v_{3}}
 !{(0.5,0) }*{\bullet}="3" !{(0.5,0.3) }*{v_{2}}
!{(1.5,0) }*{\bullet}="4"  !{(1.5,0.3) }*{v_{1}}
!{(2.5,1.5)}*{\bullet}="4a" !{(2.5,1.8) }*{v_{4}}
!{(3.5,1.5)}*{\bullet}="5a" !{(3.5,1.8) }*{v_{5}}
!{(2.5,.5)}*{\bullet}="4b"  !{(2.5,0.7) }*{v_{6}}
!{(2.5,-.5)}*{\bullet}="4c" !{(2.5,-0.3) }*{v_{7}} !{(2.5,-0.9)}*{\vdots}="vdots"
!{(2.5,-1.5)}*{\bullet}="4d" !{(2.5,-1.7) }*{v_{n+4}}
"2"-"3" "3"-"4"
"4"-"4a" "4"-"4b" "4"-"4c" "4"-"4d" "5a"-"4a" 
}
\]
\caption{$H_{3,n}$ graph}
\label{Hgraph}
\end{figure}

\subsection{Regular Spider Tree}\label{spider}
A spider graph is a tree that consists of a central vertex and several paths radiating outward from the central vertex, with no other branching. Note that in this graph, one vertex has degree at least $3$ and all others have degree at most $2$. We consider a special type of spider graph $\mathcal S_{n,m}$, where there are a total of $n$ paths, each of length $m$, connected to the central vertex. We consider the central vertex $v_1$, and label the vertices of $i^{\rm th}$ path as $\{v_i,v_{i+1},v_{i+2},\cdots, v_{i+(m-1)}\}$ for $i\in\{2,2+m,2+2m,\cdots,2+(n-1)m\}$. 

\begin{thrm}\label{spidertheorem}
Let $\mathcal S_{n,m}$ denote the spider graph with $n$ legs of length $m$. Set $A:=x-2$, and $B:=x-1$, then 
\begin{align}\label{spiderequation}
    \psi(L(\mathcal S_{n,m}); x) = \psi(Q(\mathcal S_{n,m}); x) = (x-n)C_m^n + nC_{m-1}C_m^{n-1}
\end{align}
where $\{C_k\}$ is linear non-homogeneous, recurrence relation of degree $2$, defined as
\[
C_1 = B, \quad C_2 = AB + 1, \quad C_k = A C_{k-1} + C_{k-2} \quad \text{for } k \geq 3.
\]
\end{thrm}

\begin{proof}
First observe that as $\mathcal S_{n,m}$ is a bipartite graph, $\psi(L(\mathcal S_{n,m}); x) = \psi(Q(\mathcal S_{n,m}); x)$. We establish Equation \ref{spiderequation}, by mathematical induction on $m$. First, consider the base case $m=2$. Then by considering the central vertex $v_1$, we get $\psi(L_{v_1}(\mathcal S_{n,m}); x) =((x-2)(x-1)+1)^n$, and $\psi(L_{v_1v_2}(\mathcal S_{n,m}); x) =(x-1)((x-2)(x-1)+1)^{n-1}$. Thus, using the Lemma \ref{reductionlemma} for Laplacian permanental polynomials, we obtain: $\psi(L(\mathcal S_{n,2}); x) =(x-n)C_2^n + nC_{1}C_2^{n-1}$, which establishes the base case. 

Now we consider the assumption is true for $m=k$, i.e., $\psi(L(\mathcal S_{n,k}); x) = \psi(Q(\mathcal S); x) = (x-n)C_k^n + nC_{k-1}C_k^{n-1}$. We aim to show it holds for $m = k + 1$. In the spider graph $S_{n,k+1}$, the neighbors of the central vertex $v_1$ are $\{v_{2 + i(k+1)} : 0 \leq i < n\}$.
Applying Lemma \ref{reductionlemma}, we have the Laplacian permanently polynomial of $S_{n,k+1}$ as 
$$\psi(L(\mathcal S_{n,k+1}); x) =(x-n)\psi(B_{k+1};x)^{n}+n \psi(B_{k+1};x)^{n-1} \psi(B_{k};x)$$

\noindent where $B_{k+1}$ is the tridiagonal matrix 
\[
B_{k+1} = 
\begin{bmatrix}
2 & -1 & 0 & \cdots & 0 & 0 \\
-1 & 2 & -1 & \cdots & 0 & 0 \\
0 & -1 & 2 & \cdots & 0 & 0 \\
\vdots & \vdots & \vdots & \ddots & \vdots & \vdots \\
0 & 0 & 0 & \cdots & 2 & -1 \\
0 & 0 & 0 & \cdots & -1 & 1
\end{bmatrix}.
\]
\noindent which arises as a submatrix of the Laplacian matrix $L(S_{n,k+1})$ corresponding to one of its pendant paths. It is straightforward to verify that this matrix satisfies the recurrence:
\[
\psi(B_{k+1};x)=(x-2)\psi(B_{k};x)+\psi(B_{k-1};x),
\]
which matches the recurrence for $C_{m+1}$ given in the theorem. Hence, the inductive step is verified. This completes the proof.
\end{proof}

\begin{hey}
The recurrence relation for the sequence $\{C_m\}$, defined as $C_m = A C_{m-1} + C_{m-2}$, closely resembles the Fibonacci recurrence. As a result, the number of distinct terms involved in expanding $C_m$ grows according to the Fibonacci sequence. This reflects the recursive structure of the underlying pendant paths in the spider graph $S_{n,m}$ and aligns with the growth behavior of Laplacian permanental polynomial terms as $m$ increases.
\end{hey}

\begin{thrm}
 $S_{n,2}$ is determined by its (signless) Laplacian permanental polynomial.     
\end{thrm}
\begin{proof}
Let $G$ be a graph having the same Laplacian permanental polynomial as $S_{n,2}$. Then, as derived from Theorem \ref{spidertheorem}, the polynomial admits the expansion:
\begin{align*}
     \psi(L(G); x) & =  (x-n)((x-2)(x-1)+1)^n+n(x-1)((x-2)(x-1)+1)^{n-1}\\
     & = x^{2n+1}-4nx^{2n}+\frac{15n^2-n}{2}x^{2n-1}-n(3n-2)(3n+1)x^{2n-2}+\cdots+(-3)^{n-1}(3n-1)
\end{align*}
From this Laplacian polynomial, using the Lemma \ref{firstthreevalues}, we have 
\begin{align}
    \sum_{i=0}^4 k_i = 2n+1, \quad
    \sum_{i=0}^4 i k_i = 4n, \quad
    \sum_{i=0}^4 i^2 k_i = n^2+5n, \quad
     -6t_G + \sum_{i=0}^4 i^3 k_i = n^3+9n. 
\end{align}
Consequently, we get
\begin{equation}
    6t_G - \sum_{i=0}^{d_{\max}} (i-1)(i-2)(i-n)k_i = 0. \label{4eq:key}
\end{equation}

Solving this system uniquely determines the non-negative integer solution: $t_G = 0, k_0=k_3=\cdots=k_{n-1}=0,k_1=k_2=n$, and $k_n=1$. This degree sequence corresponds to a graph with $n$ vertices of degree $1$, $n$ vertices of degree $2$, and exactly one vertex of degree $n$. There are two possible simple non-isomorphic graphs with this degree sequence. One of them is a disconnected graph with one cycle $C_n$, and a star graph $S_n$. However, the first case is impossible because the Laplacian polynomial's constant term would vanish (disconnected components), contradicting the given form. Thus, the only admissible graph is $S_{n,2}$, which is uniquely determined.
\end{proof}

\begin{hey}
The methods developed in this paper suffice to characterize $S_{n,2}$ via its Laplacian permanental polynomial. However, extending these results for $S_{n,m}$, $m\ge 3$, presents significant challenges, as the polynomial’s expansion involves higher-order terms with non-linear dependencies on $n$. Hence, it remains an open problem, and its resolution may require tools beyond those employed here.
\end{hey}

\section{Perfect Binary Tree} 

In this section, we consider the perfect binary tree $T_\ell$, which is a binary tree of depth $\ell$ having $2^i$ vertices at level $i$ for $0\le i\le \ell$. Note that for a given $\ell$, the number of vertices in $T_\ell$ is $\sum_{i=0}^\ell 2^i=2^{\ell+1}-1$. 

\begin{thrm}\label{perfectbinarythrm}
Let $T_{\ell}$ denote the perfect (complete and full) binary tree of height $\ell$. Then $$
\psi(L(T_{\ell}); x) = \psi(Q(T_\ell); x) = (x-2)A_\ell^2 + 2 A_{\ell} A_{\ell-1}^{2},
$$ where the sequence $\{A_\ell\}$ is defined recursively by
\[
A_0 = 1, \quad A_1 = x-1, \quad A_i = (x-3) A_{i-1}^2 + 2 A_{i-1}  A_{i-2}^2 \quad \text{for } m \geq 2.
\]
\end{thrm}
\begin{proof}
Since $T_{\ell}$ is bipartite, $\psi(L(T_{\ell});x)=\psi(Q(T_{\ell});x)$. We proceed with mathematical induction on $\ell$. First, we label the vertices as follows: The left branch receives $\{v_1,v_2,\dots,v_{\ell+1}\}$. The other pendant vertex connected to $v_\ell$ is $v_{\ell+2}$, and the second child of $v_{i}$ receives label $v_{2^{\ell-i}+i}$ for $0\le i\le \ell$. Subsequent labels follow this pattern. Applying Lemma \ref{reductionlemma} to $\psi(L(T_{\ell});x)$, with $v_1$ as the central vertex, we obtain: 
\begin{align}
\label{T_{ell}}
\psi(L(T_{\ell}); x) = (x-3)\psi(L_{v_1}(T_{\ell}); x)+2\psi(L_{v_1v_2}(T_{\ell}); x)
\end{align}
For the general case, the reduced matrix $L_{v_1}(T_{\ell})$ has the block diagonal form,
\[
\begin{bmatrix} 
B_\ell & 0 \\ 
0 & B_\ell 
\end{bmatrix}
\] 
\noindent where $B_\ell$ is the  square submatrix of dimension $2^{\ell}-1$ with diagonal entries consist of the values $ x-3$ and $ x-1$, appearing $2^{\ell-1}-1$ and $2^{\ell-1}$ times respectively, while the off-diagonal entries are either $0$ or $1$ (representing adjacency relations). Now when $\ell=2$, $\psi(L(B_{2}); x) =[(x-3)(x-1)^2+2(x-1)]^2$, which can easily checked as $[(x-3)A_1^2+2A_1A_0^2]$. On the other hand, $\psi(L_{v_1v_2}(T_{2}); x) =[(x-3)(x-1)^2+2(x-1)](x-1)^2 = A_2A_1^2$, which proves the base case.  

Assuming the result holds for $\ell = k$, we now establish it for $\ell = k+1$. So we have 
\begin{align}\label{perfect1}
\psi(L(T_{k}); x)=(x-2)A_k^2+2A_{k}\cdot A_{k-1}^2 
\end{align}
The Laplacian matrix $L_{v_1}(T_{k+1})$ admits the following block decomposition:
\[
\begin{bmatrix}
B_{k+1} & 0 \\
0 & B_{k+1}
\end{bmatrix}
\]
Thus, \begin{align*}
    \psi(L(T_{k+1}); x) & = (x-2)\psi(L_{v_1}(T_{k+1}); x)  + 2\psi(L{v_1v_2}(T_{k+1}); x)\\
    & = (x-2)\big(\psi(L(B_{k+1}); x)\big)^2 + 2\psi(L{v_1v_2}(T_{k+1}); x) \\
\end{align*}
A direct analysis of the matrix structure reveals that $B_{k+1}$ is nearly identical to $T_k$, except for the $(1,1)$-entry, which is $(x-3)$ instead of $(x-2)$. Consequently, the permanental polynomial takes the form (modifying Equation~\eqref{perfect1}):
\[
\psi(L(B_{k+1}); x) = (x-3)A_k^2 + 2A_k \cdot A_{k-1}^2,
\]
Meanwhile, for the edge-deleted submatrix, we have the factorization:
\[
\psi(L_{v_1v_2}(T_{k+1}); x) = \psi(L(B_{k+1}); x) \cdot \psi(L(T_k); x),
\]
This product structure reflects how the tree's symmetry affects its spectrum, which completes the proof. 
\end{proof}

Although the graph $T_2$ (with $7$ vertices) was shown to be determined by its signless Laplacian permanental polynomial computationally \cite{liu2019signless}, we present an independent theoretical proof of this result for both the Laplacian and signless Laplacian cases.

\begin{thrm}
 $T_2$ is determined by its (signless) Laplacian permanental polynomial.     
\end{thrm}
\begin{proof}
Let $G$ be a graph having the same Laplacian permanental polynomial as $T_2$. Then we have the following using Theorem \ref{perfectbinarythrm}: 
\begin{align*}\label{5eq:key}
    \psi(L(G); x) & =  (x-2)((x-3)(x-1)^2+2(x-1))^2+2 ((x-3)(x-1)^2+2(x-1)) (x-1)^2\\
    & = x^7 -12x^6 +65x^5 -200x^4 +371x^3 -408 x^2+243x-60
\end{align*}
Using the Lemma \ref{firstthreevalues}, we have 
\begin{align*}
    \sum_{i=0}^4 k_i = 7, \quad
    \sum_{i=0}^4 i k_i = 6, \quad
    \sum_{i=0}^4 i^2 k_i = 26, \quad
     -6t_G + \sum_{i=0}^4 i^3 k_i = 66. 
\end{align*}
% A key observation follows from these equations: 
% \begin{equation}
%     6t_G - \sum_{i=0}^{d_{\max}} (i-3)^2(i-1)k_i = 0. \label{4eq:key}
% \end{equation}

This system of equations admits only one non-negative integer solution: $t_G = 0, k_0=0,k_1=4, k_2=1, k_3=2$. This degree sequence corresponds to a graph with one vertex of degree $2$, two vertices of degree $3$, and finally four vertices of degree $1$. The only simple (connected) graph satisfying the given conditions is $T_2$, which is unique up to isomorphism.
% There are three possible simple graphs with this degree sequence. Two of them are illustrated in Figure \ref{K5}. We can immediately eliminate Graph $T_1$, as it is invalid due to the contradiction in the number of triangles ($t_G=1$). For Graph $T_2$, we calculate the Laplacian permanental polynomial
% \( \psi (L(T_2);x)=x^{10}-20x^{9} +185x^{8}-1030x^{7} + 3792x^{6} - 9579x^{5} - 16709x^{4} - 19762x^{3} + 15102x^{2} -6715x + 1317 \), which compared with Equation \ref{4eq:key}, shows that $\psi (L(_5\odot K_1);x)\neq \psi (L(T_2);x)$. Therefore, the only simple graph satisfying the given conditions is $C_4\odot K_1$, which is unique up to isomorphism.
\end{proof}

\begin{thrm}
$T_3$ is determined by its (signless) Laplacian permanental polynomial.     
\end{thrm}
\begin{proof}
Let $G$ be a graph having the same Laplacian permanental polynomial as $T_3$. By Theorem \ref{perfectbinarythrm}, after simplification, we obtain:
\begin{align*}
    \psi(L(G); x) & =  (x-2)[(x-3) ((x-3)(x-1)^2+2(x-1))^2 +2[(x-3)(x-1)^2+2(x-1)](x-1)^2]^2 \\ & +2((x-3) [(x-3)(x-1)^2+2(x-1)] +2[(x-3)(x-1)^2 \\ & +2(x-1)] (x-1)^2)[(x-3)(x-1)^2+2(x-1)]^2\\
    & = x^{15} -28x^{14} +371x^{13} -3074 x^{12} +\cdots-14200
\end{align*}
Applying Lemma \ref{firstthreevalues}, we derive the following system of equations:
\begin{align*}
    \sum_{i=0}^4 k_i = 15, \quad
    \sum_{i=0}^4 i k_i = 14, \quad
    \sum_{i=0}^4 i^2 k_i = 66, \quad
     -6t_G + \sum_{i=0}^4 i^3 k_i = 178. 
\end{align*}
% A key observation follows from these equations: 
% \begin{equation}
%     6t_G - \sum_{i=0}^{d_{\max}} (i-3)^2(i-1)k_i = 0. \label{4eq:key}
% \end{equation}

This system of equations admits only one non-negative integer solution: $t_G = 0, k_0=0,k_1=8, k_2=1, k_3=6$. This degree sequence corresponds to a graph with one vertex of degree $2$, six vertices of degree $3$, and finally eight vertices of degree $1$. The only connected simple graph satisfying these degree conditions is $T_3$, which is unique up to isomorphism. 
\end{proof}

\begin{hey}
Similar to the case of the Spider graph, characterizing the perfect binary tree by its (signless) Laplacian permanental polynomial follows an analogous, though considerably more tedious process. This is due to the recursive nature of computing permanental polynomials, which complicates explicit determination. As a result, we propose this as an open problem for further investigation.
\end{hey}

\section{Corona Product of $C_n$ and $K_m$} 
The \textit{corona} $G_1\odot G_2$ of two graphs is the graph obtained by taking one copy of $G_1$, and $p_1$ copies of $G_2$ (where $\vert V(G_1)\vert=p_1$), and then joining the $i^{\rm th}$ vertex of $G_1$ by an edge to every vertex in the $i^{\rm th}$ copy of $G_2$. In this section, first we consider the graph \(G=C_n\odot K_1\) that has \(2n\) vertices and \(2n\) edges. We denote the set of vertices as $ V(G) = \{ v_1, v_2, \dots, v_n, v_{n+1}, v_{n+2}, \dots, v_{2n} \},$ where the first \(n\) vertices belong to the cycle \(C_n\) (called the {inner vertices}), and the remaining \(n\) vertices are the pendant vertices. Here, each inner vertex \(v_i\) (\(1 \leq i \leq n\)) is connected to a pendant vertex \(v_{i+n}\). This corona product \(C_n\odot K_1\) is sometimes also called the \textbf{sunlight graph} in the literature. Recall $T_i(j)$ denotes the generalized (i.e., $i^{\rm {th}}$ order) Triangular number, where $$T_i(j)= \sum_{n=1}^{j} T_{i-1}(n)= \sum_{i_1=1}^i \sum_{i_2=1}^{i_1} \cdots \sum_{i_j=1}^{i_j-1} 1. $$
We observe that $T_i$ corresponds to the well-known Natural numbers, Triangular numbers, Tetrahedral numbers, Pentatope numbers for $i=0,1,2,3$ respectively, and $(i+1)$-simplex numbers for $i\ge 4$.

\begin{thrm}\label{Coronapoly}
For the graph $C_{n}\odot K_1$, 
\begin{align*}
& \psi(L(C_{n}\odot K_1); x)  =  A\psi(L_{v_1}(C_{n}\odot K_1); x)  + 2 \psi(L_{v_1v_2}(C_{n}\odot K_1); x)  + \psi(L_{v_1v_{n+1}}(C_{n}\odot K_1); x)\\ & + 2B^n, and \\
& \psi(Q(C_{n}\odot K_1); x) = A\psi(Q_{v_1}(C_{n}\odot K_1); x)  + 2 \psi(Q_{v_1v_2}(C_{n}\odot K_1); x) + \psi(Q_{v_1v_{n+1}}(C_{n}\odot K_1); x) \\ &+ (-1)^n2B^n,
\end{align*}
where $A = x-3, B=x-1$ and $k=\lfloor \frac{n-3}{2}\rfloor$, and 
\begin{align*}
&\psi(L_{v_1}(C_{n}\odot K_1); x) = \psi(Q_{v_1}(C_{n}\odot K_1); x)=(AB+1)^{n-1}B \\ & +\sum_{j=0}^k T_j(n-2-2j) (AB+1)^{n-3-2j} B^{3+2j} \\
&\psi(L_{v_1v_2}(C_{n}\odot K_1); x)=\psi(Q_{v_1v_2}(C_{n}\odot K_1); x)=B\cdot\psi(L_{v_1}(C_{n-1}\odot K_1); x)\\
&\psi(L_{v_1v_{n+1}}(C_{n}\odot K_1); x)= \psi(Q_{v_1v_{n+1}}(C_{n}\odot K_1); x)=(AB+1)^{n-1} \\&+\sum_{j=0}^k T_j(n-2-2j) (AB+1)^{n-3-2j} B^{2+2j} 
\end{align*}
\end{thrm}
\begin{proof}
We present only the proof for the Laplacian case here, noting that the signless Laplacian case follows nearly identical steps. Our argument relies on mathematical induction on $n$. While the cases for even and odd $n$ differ slightly in their details, the core argument remains analogous. Therefore, to avoid repetitive steps, we present the proof only for odd values of $n$. 

For the base case, we consider $n=3$. For the graph $C_3\odot K_1$, using its Laplacian matrix, we immediately obtain this. 
\begin{align*}
    \psi(L_{v_1}(C_{3}\odot K_1); x)  &= [(x-3)^2+1](x-1)^3 + 2(x-3)(x-1)^2+(x-1)\\
     &= [(x-3)(x-1)+1]^2(x-1)+(x-1)^3
\end{align*}
\noindent which can be viewed as $[AB+1]^2B+T_0(1)B^3$. On the other hand, 
\begin{align*}
    \psi(L_{v_1v_2}(C_{3}\odot K_1); x)  &= (x-3)(x-1)^3 + (x-1)^2\\
     &= [(x-3)(x-1)+1](x-1)
\end{align*}
matching the form $(AB+1)B$. Next,  
\begin{align*}
    \psi(L_{v_1v_4}(C_{3}\odot K_1); x)  &= [(x-3)^2+1](x-1)^2 + 2(x-3)(x-1)+ 1\\
     &= [(x-3)(x-1)+1]^2 + (x-1)^2
\end{align*}
which aligns with $(AB+1)^2 + T_0(1)B^2$. Finally, deletion of the cycle containing $v_1$ provides  $\psi(L_{v_1 v_2 v_3}(C_{3}\odot K_1); x) = (x-1)^3$, completing the base case. 

By the induction hypothesis, the claim holds for $n = m$. We extend it to $n = m+1$ as follows. First, we consider the reduced Laplacian matrix $L_{v_1}$ of $C_n\odot K_1$, obtained by deleting the row and column corresponding to a vertex $v_1$. This matrix has dimension $(2n-1)\times (2n-1)$. Its diagonal entries consist of the values $ x-3$ and $ x-1$, appearing $n-1$ and $n$ times respectively, while the off-diagonal entries are either $0$ or $1$. Consequently, the determinant of this matrix will produce terms of the form $(x-3)^a(x-1)^b$ for various combinations of $a$ and $b$. Note that in that matrix $1$ appears only in the entries represented as a pair by $\{a_{i,i+1}, a_{i,i+n}\}_{i=1}^{n-2} \cup \{a_{i,i-1}, a_{i+n,i}\}_{i=2}^{n-1}$. Hence, the existing (non-zero) permutations are only when $\sigma(i)\in \{i\pm 1,i\pm n\}$. To formalize this, let us consider two index sets.
\[
X := \{1, 2, \dots, n-1\}, \qquad Y := \{n+1, n+2, \dots, 2n\}
\]
% If $\sigma(i)=i\pm 1$ for $1\le i\le n-2$, it implies $\{i, i\pm1\}$ switches in $X$, which result in reducing the exponent of $(x-3 3)$ by two. On the other hand, if $\sigma(i)=i + n$ for $1\le i\le n-2$, or $\sigma(i)=i- n$ for $n+1 \le i\le 2n$, it implies $\{i, i+n\}$ switches (interchange between an element of $X$ and an element of $Y$), which reduces the exponents of both $(x-3)$ and $(x-1)$ by one. 

\begin{itemize}
    \item If $\sigma(i) = i + 1$ for $1 \leq i \leq n-2$, or $\sigma(i) = i -1$ for $2 \leq i \leq n-1$ this implies switches $\{i, i\pm1\}$ within $X$, which reduces the exponent of $(x-3)$ by two.
    
    \item If $\sigma(i) = i + n$ for $1 \leq i \leq n-2$, or $\sigma(i) = i - n$ for $n+1 \leq i \leq 2n$, this implies switches $\{i, i+n\}$ (interchanges between an element of $X$ and an element of $Y$), which reduces the exponents of each $(x-3)$ and $(x-1)$ by one. Note that there is no switch possible for which exponent of $(x-3)$ remains unchanged and the exponent of $(x-1)$ is reduced by $2$.
\end{itemize}

% Now, one interchange between elements of the set $X$ (and analogously between $X$ and $Y$) can occur in $n-2=T_0(n-2)$ (respectively ${n-1\choose{1}}$) distinct ways. Two interchange between elements of the set $X$ (and analogously between $X$ and $Y$) can occur in $(n-3)+(n-4)+\cdots+1=T_1(n-4)$ (respectively ${n-1\choose{2}}$) distinct ways. Continuing this any $l$ interchange took place in the set $X$ will cause $T_{l_1-1}(n-2l_1)$, for $1\le l_1\le (n-1)/2$. Similarly, any $l_2$ interchange took place between the sets $X$ and $Y$, will cause ${n-1\choose{l_2}}$, for $1\le l_2\le (n-1)$. Generalizing this, at any level, let the distinct ways it takes place be $T_j(p)\cdot {q\choose{k}}$, for some $p,q,j$, and $k$. For the next level, if the interchange takes place inside the $X$ set, that will cause moving to the left, and the interchange between $X$ and $Y$ sets will be moving right (in the Figure). For the former case, the distinct ways that can take place will be $T_{j+1}(p-2)\cdot {q-2\choose{k}}$. Similarly, the latter case, the distinct ways that can take place will be $T_j(p)\cdot {q\choose{k+1}}$. 

Now, consider interchanges between elements of the set \( X \) (and analogously between \( X \) and \( Y \)). A single interchange within \( X \) can occur in \( n-2 = T_0(n-2) \) distinct ways, while a single interchange between \( X \) and \( Y \) can occur in \( \binom{n-1}{1} \) distinct ways. For two interchanges within \( X \), the number of distinct ways is \( (n-3) + (n-4) + \cdots + 1 = T_1(n-4) \), whereas for two interchanges between \( X \) and \( Y \), it is \( \binom{n-1}{2} \). Following this pattern, any \( l_1 \) interchanges within \( X \) result in \( T_{l_1-1}(n-2l_1) \) distinct ways for \( 1 \leq l_1 \leq (n-1)/2 \), while any \( l_2 \) interchanges between \( X \) and \( Y \) yield \( \binom{n-1}{l_2} \) distinct ways for \( 1 \leq l_2 \leq n-1 \). This generates the leftmost and rightmost branches in the Figure \ref{fig:tree-structure}. 

% More generally, at $l\ge 2$ level, for $j^{\rm th}$ ($2\le j <l$)node from left, the number of distinct ways it takes place can be expressed as \( T_{p}(q) \cdot \binom{q-1}{j+1} \), where $p=l-j-1$, and $q=n-2(p-q)$. For the next level, if the interchange occurs within \( X \), the process moves left (in the figure), and the number of distinct ways becomes \( T_{p+1}(q-2) \cdot \binom{q-3}{j+1} \). Conversely, if the interchange occurs between \( X \) and \( Y \), the process moves right, and the number of distinct ways becomes \( T_p(q) \cdot \binom{q}{j+2} \). 

More generally, at level $l \geq 2$, for the $j^{\text{th}}$ node from the left ($2 \leq j < l$), the number of distinct ways it can occur is given by
\[
T_{p}(q) \cdot \binom{q-1}{j+1},
\]
where $p = l - j - 1$ and $q = n - 2(l-j)$. 

For the subsequent level, there are two possible cases:
\begin{itemize}
    \item If the interchange occurs within $X$, the process moves leftward (see the figure \ref{fig:tree-structure}), and the number of distinct ways becomes
    \[
    T_{p+1}(q-2) \cdot \binom{q-3}{j+1}
    \]
    
    \item If the interchange occurs between $X$ and $Y$, the process moves rightward, and the number of distinct ways becomes
    \[
    T_p(q) \cdot \binom{q}{j+2}
    \]
\end{itemize}
By summing all possible combinations of $(x-3)^a(x-1)^b$ as established in our previous discussion (for odd $n$), and summing branch-wise from right to left in the lattice structure, we obtain the polynomial:
\begin{align*}
    \psi(L_{v_1}(C_{n}\odot K_1); x)  &= \sum_{m=0}^n {n-1\choose m}(x-3)^{n-1}(x-1)^n \\ & + T_0(n-2)\sum_{m=0}^n {n-3\choose m}(x-3)^{n-3}(x-1)^n   \\ & 
    +T_1(n-4)\sum_{m=0}^n {n-5\choose m}(x-3)^{n-5}(x-1)^n  \cdots \\ &
    + T_{(n-3)/2}\big((x-3)^{2}(x-1)^n+(x-3)(x-1)^{n-1}+(x-1)^{n-2}\big)\\ &
    + T_{(n-1)/2}(x-1)^n \\
     &= (AB+1)^{n-1}B+\sum_{j=0}^k T_j(n-2-2j) (AB+1)^{n-3-2j} B^{3+2j} 
\end{align*}
\noindent where $A=x-3$, $B=x-1$, and $k=\lfloor \frac{n-3}{2}\rfloor$. The computation of $\psi(L_{v_1v_2}(C_{n}\odot K_1); x)$ and $\psi(L_{v_1v_{n+1}}(C_{n}\odot K_1); x)$ follows a similar pattern to our previous analysis. 
\begin{itemize}
    \item For $L_{v_1v_2}$, the lattice structure mirrors Figure~\ref{fig:tree-structure} but with the initial node shifted to $(n-2,n)$. 
    \item For $L_{v_1v_{n+1}}$, the structure undergoes a different shift, now beginning at $(n-1,n-1)$.
\end{itemize}
The argument is complete once we account for the different starting nodes, as the computational method remains otherwise unchanged.
\end{proof}

\begin{figure}[h!]
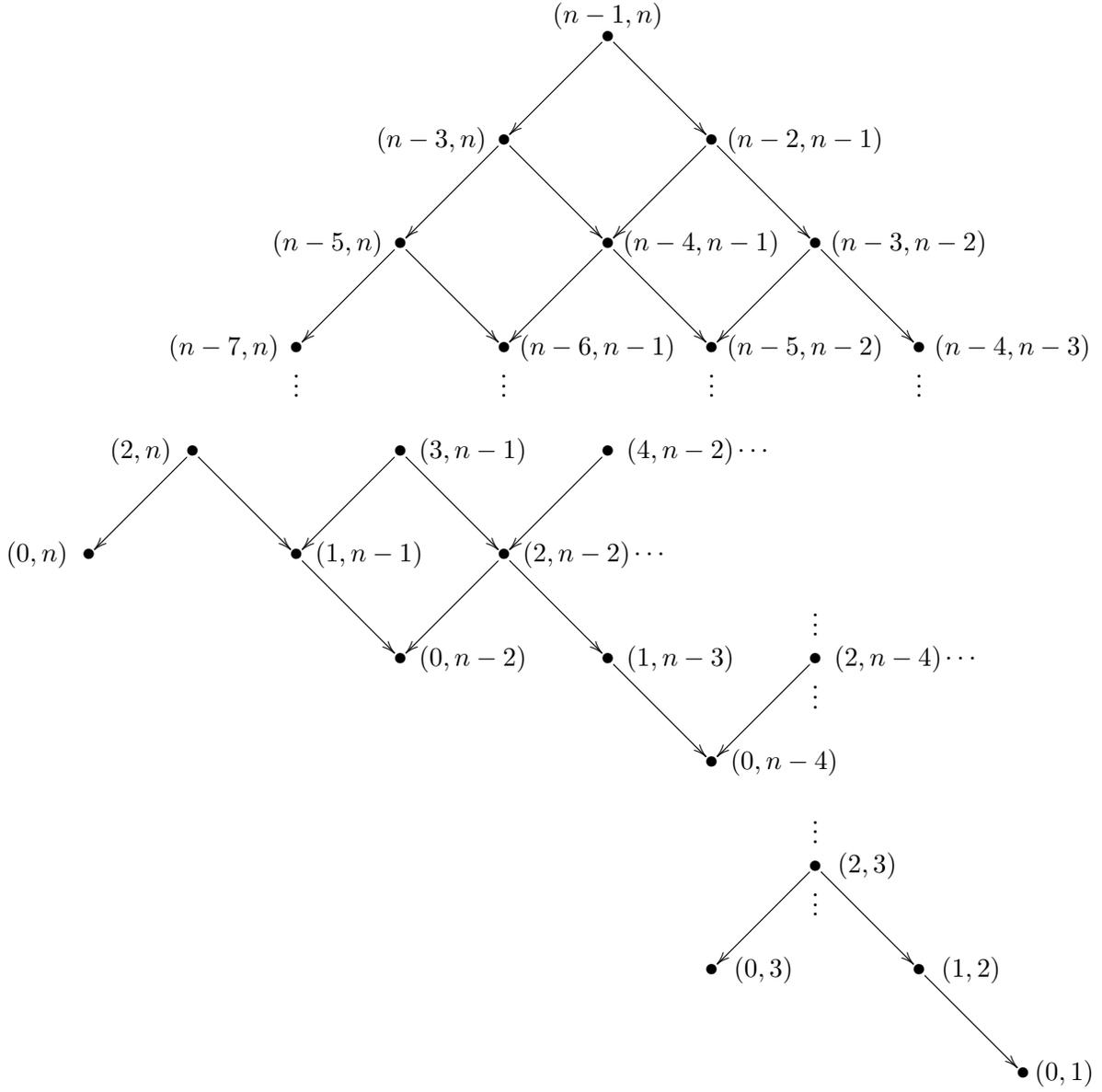

\[
\xygraph{
!{<0cm,0cm>;<1.5 cm,0 cm>:<0 cm,1.5 cm>::}
% Level 0
!{(0,0)}*{\bullet}="A11" !{(0,0.2)}*{(n-1,n)}
% Level 1
!{(-1,-1)}*{\bullet}="A21" !{(-1.7,-1)}*{(n-3,n)}
!{(1,-1)}*{\bullet}="A22" !{(1.9,-1)}*{(n-2,n-1)}
% Level 2
!{(-2,-2)}*{\bullet}="A31" !{(-2.7,-2)}*{(n-5,n)}
!{(0,-2)}*{\bullet}="A32" !{(0.9,-2)}*{(n-4,n-1)}
!{(2,-2)}*{\bullet}="A33" !{(2.9,-2)}*{(n-3,n-2)}
% Level 3
!{(-3,-3)}*{\bullet}="A41" !{(-3.7,-3)}*{(n-7,n)} !{(-3,-3.3)}*{\vdots}="vdots"
!{(-1,-3)}*{\bullet}="A42" !{(-.1,-3)}*{(n-6,n-1)} !{(-1,-3.3)}*{\vdots}="vdots"
!{(1,-3)}*{\bullet}="A43" !{(1.9,-3)}*{(n-5,n-2)}  !{(1,-3.3)}*{\vdots}="vdots"
!{(3,-3)}*{\bullet}="A44" !{(3.9,-3)}*{(n-4,n-3)}  !{(3,-3.3)}*{\vdots}="vdots"
% Level 4
!{(-4,-4)}*{\bullet}="B11" !{(-4.5,-4)}*{(2,n)}
!{(-2,-4)}*{\bullet}="B12" !{(-1.3,-4)}*{(3,n-1)}
!{(0,-4)}*{\bullet}="B13" !{(0.7,-4)}*{(4,n-2)} !{(1.4,-4)}*{\cdots }="dots"
!{(-5,-5)}*{\bullet}="B21" !{(-5.5,-5)}*{(0,n)}
!{(-3,-5)}*{\bullet}="B22" !{(-2.3,-5)}*{(1,n-1)}
!{(-1,-5)}*{\bullet}="B23" !{(-0.3,-5)}*{(2,n-2)} !{(0.4,-5)}*{\cdots }="dots"
!{(-2,-6)}*{\bullet}="B31" !{(-1.3,-6)}*{(0,n-2)}
!{(0,-6)}*{\bullet}="B32" !{(0.7,-6)}*{(1,n-3)}
!{(2,-6)}*{\bullet}="B33" !{(2.7,-6)}*{(2,n-4)} !{(2,-5.6)}*{\vdots}="vdots" 
!{(2,-6.3)}*{\vdots}="vdots" !{(3.4,-6)}*{\cdots }="dots"
!{(1,-7)}*{\bullet}="B41" !{(1.7,-7)}*{(0,n-4)}
% Lower branch (2,3) to (0,1)
!{(2,-8)}*{\bullet}="C11" !{(2.5,-8)}*{(2,3)} !{(2,-8.3)}*{\vdots}="vdots"
!{(2,-7.6)}*{\vdots}="vdots"
!{(1,-9)}*{\bullet}="C21" !{(1.5,-9)}*{(0,3)}
!{(3,-9)}*{\bullet}="C22" !{(3.5,-9)}*{(1,2)}
!{(4,-10)}*{\bullet}="C31" !{(4.4,-10)}*{(0,1)}
% Edges
"A11": "A21" "A11": "A22"
"A21": "A31" "A21":"A32"  "A22":"A32"  "A22":"A33"  
"A31": "A41" "A31":"A42"  "A32":"A42"  "A32":"A43"   "A33":"A43"  "A33":"A44"
%"A41"-@{-->} "B11" "A41"-@{-->} "B12" "A42"-@{-->} "B12" "A42"-@{-->} "B13"  
"B11": "B21" "B11":"B22"  "B12":"B22"   "B12":"B23"  "B13":"B23" 
"B22":"B31" "B23":"B31" "B23":"B32"  "B32":"B41"  "B33":"B41" 
"C11": "C21" "C11": "C22" "C22": "C31" 
}
\]
\caption{Recursive Lattice structure showing decomposition from $(n-1,n)$, $n$ odd.}
\label{fig:tree-structure}
\end{figure}

% \begin{thrm}\label{Coronapolysignless}
% For the graph $C_{n}\odot K_1$ graph $$
% \psi(Q(C_{n}\odot K_1); x) = A\psi(Q_{v_1}(C_{n}\odot K_1); x)  + 2 \psi(Q_{v_1v_2}(C_{n}\odot K_1); x) + \psi(Q_{v_1v_{n+1}}(C_{n}\odot K_1); x) + (-1)^nB^n,
% $$ where $A = x-3, B=x-1$ and $k=\lfloor \frac{n-3}{2}\rfloor$, and 
% \[
% \psi(L_{v_1}(C_{n}\odot K_1); x)=(AB+1)^{n-1}B+\sum_{j=0}^k T_j(n-2-2j) (AB+1)^{n-3-2j} B^{3+2j} 
% \]
% \[
% \psi(L_{v_1v_2}(C_{n}\odot K_1); x)=B\cdot\psi(L_{v_1}(C_{n-1}\odot K_1); x)
% \]
% \[
% \psi(L_{v_1v_{n+1}}(C_{n}\odot K_1); x)=(AB+1)^{n-1}+\sum_{j=0}^k T_j(n-2-2j) (AB+1)^{n-3-2j} B^{2+2j} 
% \]
% \end{thrm}

While the characterization of $C_3 \odot K_1$ was previously known through computation \cite{liu2019signless}, we offer an independent theoretical proof.

\begin{thrm}
 $C_3\odot K_1$ is determined by its (signless) Laplacian permanental polynomial.     
\end{thrm}
\begin{proof}
First, we establish the result for the Laplacian permanental polynomial. Let $G$ be a graph having the same Laplacian permanental polynomial as $C_4\odot K_1$. Then by Theorem \ref{Coronapoly} together with Lemma \ref{firstthreevalues} we have the following: 
\begin{align}
    \psi(L(G); x) =  x^6-12x^5+63x^4-176x^3 +273x^2 -215x+74
\end{align}
Using the Lemma \ref{firstthreevalues}, we have 
\begin{align*}
    \sum_{i=0}^3 k_i = 6, \quad
    \sum_{i=0}^3 i k_i = 12, \quad\quad
    \sum_{i=0}^3 i^2 k_i = 30, \quad
     -6t_G + \sum_{i=0}^3 i^3 k_i = 78. 
\end{align*}
% A key observation follows from these equations: 
% \begin{equation}
%     6t_G - \sum_{i=0}^{d_{\max}} (i-3)^2(i-1)k_i = 0. \label{C3eq:key}
% \end{equation}

This system of equations admits two non-negative integer solutions: $t_G = 1, k_0=k_2=0,k_1=k_3=3$, and $t_G = 0, k_0=1, k_1=0,k_2=3, k_3=2$. However, since the constant term of $\psi(L(G);x)$ is not zero, the graph $G$ cannot have isolated vertices, which eliminates the second solution, which requires $k_0=1$. Thus, the only feasible solution is the first one, where $t_G=1$ and the degree sequence consists of three vertices of degree $1$ and three vertices of degree $3$. The graph $ C_3\odot K_1$ is the unique graph satisfying these conditions, up to isomorphism. 
Next, the signless Laplacian permanental polynomial of $G$ is as follows,  
\begin{align}
    \psi(Q(G); x) =  x^6-12x^5+63x^4-180x^3 +285x^2 -227x+78
\end{align}
Similar arguments lead us to the same conclusions as the Laplacian permanental polynomial. 
\end{proof}

\begin{thrm}
 $C_4\odot K_1$ is determined by its (signless) Laplacian permanental polynomial.     
\end{thrm}
\begin{proof}
Let $G$ be any graph sharing the Laplacian permanental polynomial with $C_4 \odot K_1$  (Theorem \ref{Coronapoly}). From the Lemma \ref{firstthreevalues} we have the following: 
\begin{align}
    \psi(L(G); x) =  x^8-16x^7+116x^6 -488x^5+1288x^4-2144x^3+2190 x^2-1280x+324
\end{align}
Using the Lemma \ref{firstthreevalues}, we have 
\begin{align*}
    \sum_{i=0}^4 k_i = 8, \quad
    \sum_{i=0}^4 i k_i = 16, \quad
    \sum_{i=0}^4 i^2 k_i = 40, \quad
     -6t_G + \sum_{i=0}^4 i^3 k_i = 112. 
\end{align*}
% A key observation follows from these equations: 
% \begin{equation}
%     6t_G - \sum_{i=0}^{d_{\max}} (i-3)^2(i-1)k_i = 0. \label{4eq:key}
% \end{equation}

This system of equations admits only one non-negative integer solution: $t_G = 0, k_0=k_2=0,k_1=k_3=4$. This degree sequence corresponds to a graph with four vertices of degree $1$ and four vertices of degree $3$. There are two possible simple graphs with this degree sequence. One of them is illustrated in Figure \ref{H1graph}. But in this graph, $t_G=1$, which makes it invalid due to the contradiction in the number of triangles. Therefore, the only simple graph satisfying the given conditions is $C_4\odot K_1$, which is unique up to isomorphism. The signless Laplacian permanental polynomial is identical to the Laplacian case; hence, an identical argument applies. 
\end{proof}

\begin{figure}[h!]
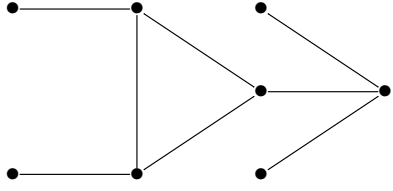

\centering
\[
\xygraph{
  !{<0cm,0cm>;<1.1cm,0cm>:<0cm,1.1cm>::}
  % vertices
  !{(0.5,2)}*{\bullet}="TL"      % top‑left of rectangle
  !{(2,2)}*{\bullet}="TR"      % top‑right
  !{(2,0)}*{\bullet}="BR"      % bottom‑right
  !{(0.5,0)}*{\bullet}="BL"      % bottom‑left
  !{(3.5,1)}*{\bullet}="A"       % apex of the triangle
  !{(5,1)}*{\bullet}="M"       % middle of right path
  !{(3.5,2)}*{\bullet}="T"       % top of right path
  !{(3.5,0)}*{\bullet}="B"       % bottom of right path
  % rectangle
  "TL"-"TR"-"BR"-"BL"
  % triangle
  "TR"-"A"-"BR"
  % link to the path
  "A"-"M"
  % vertical path of length 2
  "M"-"T"
  "M"-"B"
}
\]
\caption{A possible simple graph with the degree sequence of $3,3,3,3,1,1,1,1$.}
\label{H1graph}
\end{figure}

\begin{thrm}
 $C_5\odot K_1$ is determined by its (signless) Laplacian permanental polynomial.     
\end{thrm}
\begin{proof}
Let $G$ be a graph having the same Laplacian permanental polynomial as $C_5\odot K_1$ (Corollary \ref{Coronapoly}). From the Lemma \ref{firstthreevalues} we have the following: 
\begin{align}\label{4eq:key123}
    \psi(L(G); x) =  x^8-16x^7+116x^6 -488x^5+1288x^4-2144x^3+2190 x^2-1280x+324
\end{align}
Using the Lemma \ref{firstthreevalues}, we have 
\begin{align*}
    \sum_{i=0}^4 k_i = 10, \quad
    \sum_{i=0}^4 i k_i = 20, \quad
    \sum_{i=0}^4 i^2 k_i = 50, \quad
     -6t_G + \sum_{i=0}^4 i^3 k_i = 140. 
\end{align*}
% A key observation follows from these equations: 
% \begin{equation}
%     6t_G - \sum_{i=0}^{d_{\max}} (i-3)^2(i-1)k_i = 0. \label{4eq:key}
% \end{equation}

This system of equations admits only one non-negative integer solution: $t_G = 0, k_0=k_2=0,k_1=k_3=5$. This degree sequence corresponds to a graph with five vertices of degree $1$ and five vertices of degree $3$. There are three possible simple graphs with this degree sequence. Two of them are illustrated in Figure \ref{K5}. We can immediately eliminate Graph $T_1$, as it is invalid due to the contradiction in the number of triangles ($t_G=1$). For Graph $T_2$, we calculate the Laplacian permanental polynomial
\( \psi (L(T_2);x)=x^{10}-20x^{9} +185x^{8}-1030x^{7} + 3792x^{6} - 9579x^{5} - 16709x^{4} - 19762x^{3} + 15102x^{2} -6715x + 1317 \), which compared with Equation \ref{4eq:key123}, shows that $\psi (L(C_5\odot K_1);x)\neq \psi (L(T_2);x)$. Therefore, the only simple graph satisfying the given conditions is $C_4\odot K_1$, which is unique up to isomorphism. For the signless Laplacian case, a very similar argument applies. 
\end{proof}

\begin{figure}[h!]
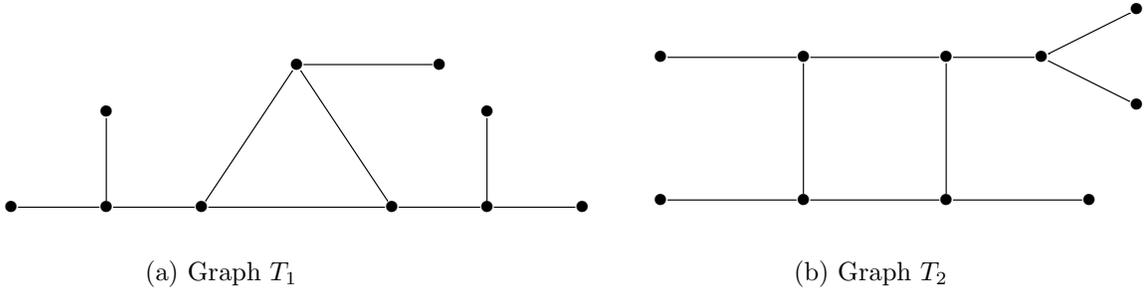

    \centering
     \begin{subfigure}[b]{0.35\textwidth}
         \[ \xygraph{
 !{(-7,-1)}*{\bullet}="1"      % top‑left of rectangle
  !{(-6,-1)}*{\bullet}="3"      % top‑right
  !{(-4,.5)}*{\bullet}="4"      % bottom‑right
  !{(-6,0)}*{\bullet}="2"      % bottom‑left
  !{(-5,-1)}*{\bullet}="5"       % apex of the triangle
  !{(-3,-1)}*{\bullet}="6"       % apex of the triangle
  !{(-2.5,.5)}*{\bullet}="7"       % middle of right path
  !{(-2,-1)}*{\bullet}="8"       % top of right path
  !{(-2,0)}*{\bullet}="9"       % bottom of right path
  !{(-1,-1)}*{\bullet}="10"
  % rectangle
  "1"-"3"-"5"-"6"-"8"-"10" "2"-"3" "5"-"4"-"7" "4"-"6" "8"-"9"
} 
\]    
     \subcaption{Graph $T_1$}
     \end{subfigure}
     \hspace{2.6 cm}
     \begin{subfigure}[b]{0.35\textwidth}
        \[ \xygraph{
 !{(0.5,2.5)}*{\bullet}="1a"      % top‑left of rectangle
  !{(2,2.5)}*{\bullet}="3a"      % top‑right
  !{(2,1)}*{\bullet}="4a"      % bottom‑right
  !{(0.5,1)}*{\bullet}="2a"      % bottom‑left
  !{(3.5,2.5)}*{\bullet}="5a"       % apex of the triangle
  !{(3.5,1)}*{\bullet}="6a"       % apex of the triangle
  !{(4.5,2.5)}*{\bullet}="7a"       % middle of right path
  !{(5.5,3)}*{\bullet}="8a"       % top of right path
  !{(5.5,2)}*{\bullet}="9a"       % bottom of right path
  !{(5,1)}*{\bullet}="10a"
  % rectangle
  "1a"-"3a"-"5a"-"7a" -"8a"
  % triangle
  "3a"-"4a"-"2a" "4a"-"6a"-"5a" "6a"-"10a" "7a"-"9a"
} 
\]
     \subcaption{Graph $T_2$}
     \end{subfigure}
     \caption{Two possible non-isomorphic simple graphs with the same degree sequence of $3,3,3,3,3,1,1,1,1,1$.}
        \label{K5}
\end{figure}

% {\color{blue} SOON THE NUMBER OF NON-ISOMORPHIC GRAPHS WITH THE SAME DEGREE SEQUENCE EXPLODES. FOR EXAMPLE $C_6\odot K_1$ HAS 3 SUCH GRAPHS}

It can be observed that as $n$ increases, the number of non-isomorphic graphs sharing the same degree sequence grows rapidly. For instance, the graph $C_6\odot K_1$ admits three other non-isomorphic graphs with the same degree sequence. This motivates the following conjecture.

\begin{conj} 
For $n\ge 6$, the $C_n\odot K_1$ is determined by its (signless) Laplacian permanental polynomials.     
\end{conj}

Next, we consider the family of graphs $C_m\odot\bar K_n$ for various values of $m$. Let us label the vertices of $V(C_m\odot\bar K_n)$ as follows: $\{v_1,v_2,\cdots,v_m\}$ are the vertices of the cycle $C_m$, and $ \{v_{m+(i-1)n+1},v_{m+(i-1)n+2},\cdots, v_{m+in}\}$ are the pendant vertices connected with $v_i$ for $1\le i\le m$. 

\begin{thrm}\label{C3corona-bark}
For the $G=C_{3}\odot \bar K_n$ graph 
\begin{align*}
& \psi(L(G); x)  =  (A^3+3A+2)B^{3n} +3n(A^2+1)B^{3n-1} +3n^3AB^{3n-2} +n^3B^{3n-3},\\
& \psi(Q(G); x)  =  (A^3+3A-2)B^{3n} +3n(A^2+1)B^{3n-1} +3n^3AB^{3n-2} +n^3B^{3n-3},
\end{align*}
where $A = x-(n+2)$, and $B=x-1$ .
\end{thrm}

\begin{proof}
Let us first consider the Laplacian case. The reduced matrix $L_{v_1}$ of $G$, obtained by deleting the row and column corresponding to vertex $v_1$, has dimension $n(m+1)-1$. The diagonal entries consist of $x-(n+2)$ appearing twice and $x-1$ appearing $n(m+1)-3$ times, while off-diagonal entries are either $0$ or $1$. This structure leads to a determinant producing terms of the form $(x-3)^a(x-1)^b$, following an approach similar to Theorem~\ref{Coronapoly}. Since the entry $1$ appears only in positions $\{a_{i,3+in+k}, a_{3+in+k,i}\}_{k=0}^{n-1}$ for $i=1,2$, the non-zero permutations occur precisely when $\sigma(i) \in \{3+in+k\}_{k=0}^{n-1}$ for $i=1,2$. Defining index sets $X := \{1, 2\}$ and $Y := \{2+n+1,\dots,2+3n\}$, we observe that: a single interchange within $X$ has one possibility; a single interchange between $X$ and $Y$ has $2n$ possibilities; and two interchanges (between $X$ and $Y$) have $n^2$ possibilities (see Figure \ref{fig:tree-structureC3corona}). This yields the characteristic polynomial:
\begin{align}\label{c2v1laplacian}
    \psi(L_{v_1}(G); x)  = A^2B^{3n}+2nAB^{3n-1} + n^2B^{3n-2} + B^{3n}
\end{align} 

Similarly, we observe that both $L_{v_1v_2}(G)$ and $L_{v_1v_4}(G)$ exhibit analogous lattice structures that decompose from the nodes $(1,3n)$ and $(2,3n-1)$, respectively. For the graph $G$, we obtain the following characteristic polynomials:

\begin{align}
\psi(L_{v_1v_2}(G); x) &= AB^{3n} + nB^{3n-1}, \label{c2v1v2laplacian} \\
\psi(L_{v_1v_4}(G); x) &= A^2B^{3n-1} + 2nAB^{3n-2} + n^2B^{3n-3} + B^{3n-1} \label{c2v1v4laplacian}
\end{align}
Furthermore, the characteristic polynomial for the entire cycle deletion that contains the $v_1$ is simply $\psi(L_{v_1v_2v_3}(G); x) = B^{3n}$. Combining these results (Equation \ref{c2v1laplacian},\ref{c2v1v2laplacian}, and \ref{c2v1v4laplacian}), we derive the complete Laplacian characteristic polynomial:
\begin{align*}
    \psi(L(G); x)  &= (x-(n+2)) \psi(L_{v_1}(G); x) + 2 \psi(L_{v_1v_2}(G); x) +n \psi(L_{v_1v_4}(G); x) + 2\psi(L_{v_1v_2v_3}(G); x)\\ 
    &  =A (A^2B^{3n}+2nAB^{3n-1} + n^2B^{3n-2} + B^{3n}) + 2 (AB^{3n} + nB^{3n-1})+n(A^2B^{3n-1} \\ &+ 2nAB^{3n-2} + n^2B^{3n-3} + B^{3n-1}) +2B^{3n} \\ 
    & =(A^{3}+3A+2) B^{3n} + 3n(A^{2}+1) B^{3n-1} + 3n^2AB^{3n-2} + n^3 B^{3n-3}
\end{align*}
The signless Laplacian case follows similarly. 

\end{proof}

% \begin{figure}[h!]
% \[
% \xygraph{
% !{<0cm,0cm>;<1.5 cm,0 cm>:<0 cm,1.5 cm>::}
% % Level 0
% !{(0,0)}*{\bullet}="A11" !{(0,0.2)}*{(2,3n)}
% !{(0,-0.2)}*[blue]{1}
% % Level 1
% !{(-1,-1)}*{\bullet}="A21" !{(-1.5,-1)}*{(0,3n)}
% !{(-.8,-1)}*[blue]{1}
% !{(1,-1)}*{\bullet}="A22" !{(1.8,-1)}*{(1,3n-1)}
% !{(.7,-1)}*[blue]{2n}
% % Level 2
% !{(2,-2)}*{\bullet}="A33" !{(2.7,-2)}*{(0,3n-2)}
% !{(1.7,-2)}*[blue]{n^2}
% "A11": "A21" "A11": "A22" "A22":"A33"  
% }
% \]
% \caption{Recursive Lattice structure showing decomposition from $(2,3n)$ for $C_{3}\odot \bar K_n$}
% \label{fig:tree-structureC3corona}
% \end{figure}

\begin{figure}[h!]
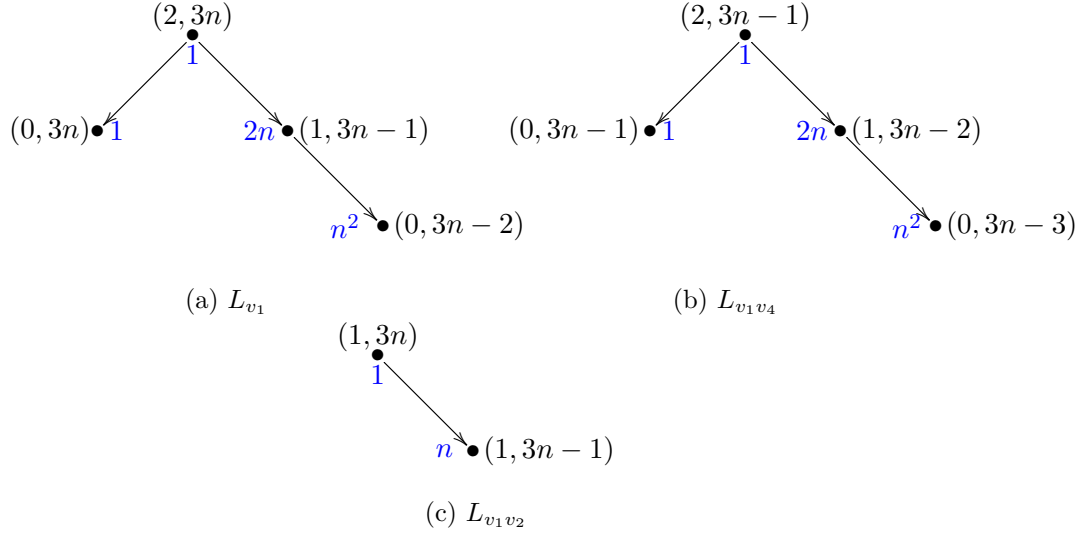

    \centering
     \begin{subfigure}[b]{0.35\textwidth}
         \[ \xygraph{
!{(0,0)}*{\bullet}="A11" !{(0,0.2)}*{(2,3n)}
!{(0,-0.2)}*[blue]{1}
% Level 1
!{(-1,-1)}*{\bullet}="A21" !{(-1.5,-1)}*{(0,3n)}
!{(-.8,-1)}*[blue]{1}
!{(1,-1)}*{\bullet}="A22" !{(1.8,-1)}*{(1,3n-1)}
!{(.7,-1)}*[blue]{2n}
% Level 2
!{(2,-2)}*{\bullet}="A33" !{(2.8,-2)}*{(0,3n-2)}
!{(1.6,-2)}*[blue]{n^2}
"A11": "A21" "A11": "A22" "A22":"A33"  
} 
\]    
     \subcaption{$L_{v_1}$}
     \end{subfigure}
     \hspace{.6 cm}
     \begin{subfigure}[b]{0.35\textwidth}
        \[ \xygraph{
!{(0,0)}*{\bullet}="A11" !{(0,0.2)}*{(2,3n-1)}
!{(0,-0.2)}*[blue]{1}
% Level 1
!{(-1,-1)}*{\bullet}="A21" !{(-1.8,-1)}*{(0,3n-1)}
!{(-.8,-1)}*[blue]{1}
!{(1,-1)}*{\bullet}="A22" !{(1.8,-1)}*{(1,3n-2)}
!{(.7,-1)}*[blue]{2n}
% Level 2
!{(2,-2)}*{\bullet}="A33" !{(2.8,-2)}*{(0,3n-3)}
!{(1.7,-2)}*[blue]{n^2}
"A11": "A21" "A11": "A22" "A22":"A33"  
} 
\]
     \subcaption{$L_{v_1v_4}$}
     \end{subfigure}
       \hspace{.6 cm}
     \begin{subfigure}[b]{0.35\textwidth}
        \[ \xygraph{
!{(0,0)}*{\bullet}="A11" !{(0,0.2)}*{(1,3n)}
!{(0,-0.2)}*[blue]{1}
% Level 1
!{(1,-1)}*{\bullet}="A22" !{(1.8,-1)}*{(1,3n-1)}
!{(.7,-1)}*[blue]{n}
"A11": "A22" 
} 
\]
     \subcaption{$L_{v_1v_2}$}
     \end{subfigure}
     \caption{Recursive Lattice structure showing decomposition for $C_{3}\odot \bar K_n$}
        \label{fig:tree-structureC3corona}
\end{figure}

\begin{thrm}
 $C_{3}\odot \bar K_n$ is determined by its (signless) Laplacian permanental polynomial.     
\end{thrm}
\begin{proof}
Let $G$ be a graph having the same Laplacian permanental polynomial as $C_{3}\odot \bar K_n$. Then, from Theorem \ref{C3corona-bark}, after a tedious derivation, we get the following polynomial: 
\begin{align*}
     \psi(L(G); x) & =  x^{3n+3} -(6n+6)x^{3n+2} +\frac{33n^2+63n+30}{2}x^{3n+1} -(28n^3+75n^2+61n+12) x^{3n}\\ &+\cdots+ (1)^{n-1}(8n^3+24n^2+30n+12)
\end{align*}
From this Laplacian polynomial, using the Lemma \ref{firstthreevalues}, we have 
\begin{align*}
    \sum_{i=0}^4 k_i = 3n+3, \quad
    \sum_{i=0}^4 i k_i = 3n+3, \quad
    \sum_{i=0}^4 i^2 k_i = 3n^2+17n +12, \quad & \\ 
     -6t_G + \sum_{i=0}^4 i^3 k_i = 3n^3 + 18n^2 + 39n+ 18. 
\end{align*}
 A key observation follows from these equations: 
\begin{equation}
    6t_G - \sum_{i=0}^{d_{\max}} (i-1)(i-n-2)k_i = 0. \label{4eq:key}
\end{equation}

Solving this system uniquely determines the non-negative integer solution: $t_G = 1, k_0=k_2=\cdots=k_{n+1}=0,k_1=3n$, and $k_{n+2}=3$. This degree sequence describes a graph with $3n$ vertices of degree $1$, and three vertices of degree $n+2$. The three vertices of degree $n+2$ must form a triangle, each adjacent to $n$ pendant vertices (degree $1$). Consequently, the only graph realizing this degree sequence is $C_{3}\odot \bar K_n$, which is uniquely determined. For the signless Laplacian, an analogous procedure follows.
\end{proof}

\begin{thrm}\label{C4corona-bark}
For the graph $C_{4}\odot \bar K_n$, 
\begin{align*}
& \psi(L(C_{4}\odot \bar K_n); x)  =  (A^2+2)^2B^{4n} +2nA (2A^2+3)B^{4n-1} + 2n^2(3A^2+1) + 4n^3AB^{4n-3} +n^4B^{4n-4}\\
& \psi(Q(C_{4}\odot \bar K_n); x)  =  A^2(A^2+4)B^{3n} +2nA (2A^2+3)B^{4n-1} + 2n^2(3A^2+1) + 4n^3AB^{4n-3} +n^4B^{4n-4}
\end{align*}

where $A = x-(n+2)$, and $B=x-1$.
\end{thrm}
\begin{proof}
Consider the reduced Laplacian matrix $L_{v_1}$ of $G= C_{4}\odot \bar K_n$, obtained by removing the row and column corresponding to the vertex $v_1$. As it is very similar to the previous Theorem, we only need to identify that for index sets $X := \{1, 2, 3\}$ and $Y := \{3+n+1,\dots,3+4n\}$, which yields the characteristic polynomial (see the Lattice structure in Figure \ref{fig:tree-structureC4corona}):
\begin{align}\label{c4v1laplacian}
    \psi(L_{v_1}(G); x)  = A^3B^{4n}+3n A^2 B^{4n-1} +3n^2 A B^{4n-2} + n^3B^{4n-3} + 2A B^{4n}+ 2n B^{4n-1}
\end{align} 
Similarly, we observe that both $L_{v_1v_2}(G)$ and $L_{v_1v_5}(G)$ exhibit analogous lattice structures that decompose from the nodes $(2,4n)$ and $(3,4n-1)$, respectively. For the graph $G$, we obtain the following characteristic polynomials:

\begin{align}
\psi(L_{v_1v_2}(G); x) &= A^2 B^{4n} + 2n A B^{4n-1} + n^2 B^{4n-2} + B^{4n} , \label{c4v1v2laplacian} \\
\psi(L_{v_1v_5}(G); x) &= A^3 B^{4n-1} + 3n A^2 B^{4n-2} + 3n^2 A B^{4n-3} +n^3 B^{4n-4} \label{c4v1v5laplacian}
\end{align}
Furthermore, the characteristic polynomial for the entire cycle deletion that contains the $v_1$ is simply $\psi(L_{v_1v_2v_3v_4}(G); x) = B^{4n}$. Combining these results (Equation \ref{c4v1laplacian},\ref{c4v1v2laplacian}, and \ref{c4v1v5laplacian}), we derive the complete Laplacian characteristic polynomial:
\begin{align*}
    \psi(L(G); x)  &= (x-(n+2)) \psi(L_{v_1}(G); x) + 2 \psi(L_{v_1v_2}(G); x) +n \psi(L_{v_1v_5}(G); x) + 2\psi(L_{v_1v_2v_3}(G); x)\\ 
    & = A (A^3B^{4n}+3n A^2 B^{4n-1} +3n^2 A B^{4n-2} + n^3B^{4n-3} + 2A B^{4n}+ 2n B^{4n-1}  \\  & + 2 (A^2 B^{4n} + 2n A B^{4n-1} + n^2 B^{4n-2} + B^{4n}) \\ 
    & + n(A^3 B^{4n-1} + 3n A^2 B^{4n-2} + 3n^2 A B^{4n-3} +n^3 B^{4n-4}) +2B^{4n} \\ 
    & =(A^{4}+4A^2+4) B^{4n} + (4nA^{3}+6nA) B^{4n-1} + (6n^2 A^2 +2n^2) B^{4n-2} + 4n^3 A B^{4n-3} \\ & +n^4B^{4n-4}\\
    & =(A^2+2)^2B^{4n} +2nA (2A^2+3)B^{4n-1} + 2n^2(3A^2+1) + 4n^3AB^{4n-3} +n^4B^{4n-4}
\end{align*}
The case for the signless Laplacian is treated similarly. 
\end{proof}

% \begin{figure}[h!]
% \[
% \xygraph{
% !{<0cm,0cm>;<1.5 cm,0 cm>:<0 cm,1.5 cm>::}
% % Level 0
% !{(0,0)}*{\bullet}="A11" !{(0,0.2)}*{(3,4n)}
% !{(0,-0.2)}*[blue]{1}
% % Level 1
% !{(-1,-1)}*{\bullet}="A21" !{(-1.5,-1)}*{(1,4n)}
% !{(-0.8,-1)}*[blue]{1}
% !{(1,-1)}*{\bullet}="A22" !{(1.7,-1)}*{(2,4n-1)}
% !{(0.7,-1)}*[blue]{3n}
% % Level 2
% !{(0,-2)}*{\bullet}="A32" !{(0.8,-2)}*{(0,4n-1)}
% !{(0,-1.7)}*[blue]{2n}
% !{(2,-2)}*{\bullet}="A33" !{(2.8,-2)}*{(1,4n-2)}
% !{(1.8,-2.2)}*[blue]{3n^2}
% % Level 3
% !{(3,-3)}*{\bullet}="A44" !{(3.7,-3)}*{(0,4n-3)}  
% !{(2.7,-3)}*[blue]{n^3}  
% "A11": "A21" "A11": "A22"
% "A21":"A32"  "A22":"A32"  "A22":"A33"  
% "A33":"A44"
% }
% \]
% \caption{Recursive Lattice structure showing decomposition from $(3,4n)$ for $C_{4}\odot \bar K_n$}
% \label{fig:tree-structureC4corona}
% \end{figure}

\begin{figure}[h!]
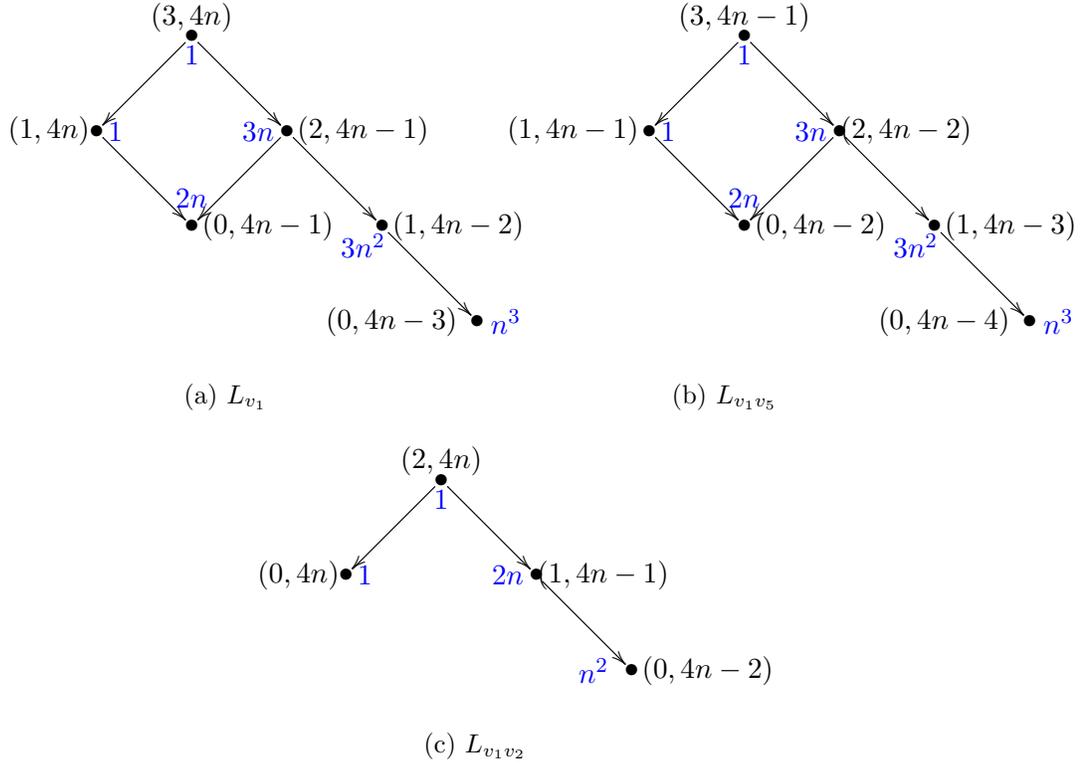

    \centering
     \begin{subfigure}[b]{0.35\textwidth}
         \[ \xygraph{
% Level 0
!{(0,0)}*{\bullet}="A11" !{(0,0.2)}*{(3,4n)}
!{(0,-0.2)}*[blue]{1}
% Level 1
!{(-1,-1)}*{\bullet}="A21" !{(-1.5,-1)}*{(1,4n)}
!{(-0.8,-1)}*[blue]{1}
!{(1,-1)}*{\bullet}="A22" !{(1.8,-1)}*{(2,4n-1)}
!{(0.7,-1)}*[blue]{3n}
% Level 2
!{(0,-2)}*{\bullet}="A32" !{(0.8,-2)}*{(0,4n-1)}
!{(0,-1.7)}*[blue]{2n}
!{(2,-2)}*{\bullet}="A33" !{(2.8,-2)}*{(1,4n-2)}
!{(1.8,-2.2)}*[blue]{3n^2}
% Level 3
!{(3,-3)}*{\bullet}="A44" !{(2.1,-3)}*{(0,4n-3)}  
!{(3.3,-3)}*[blue]{n^3}  
"A11": "A21" "A11": "A22"
"A21":"A32"  "A22":"A32"  "A22":"A33"  
"A33":"A44"
} 
\]    
     \subcaption{$L_{v_1}$}
     \end{subfigure}
     \hspace{.6 cm}
     \begin{subfigure}[b]{0.35\textwidth}
        \[ \xygraph{
% Level 0
!{(0,0)}*{\bullet}="A11" !{(0,0.2)}*{(3,4n-1)}
!{(0,-0.2)}*[blue]{1}
% Level 1
!{(-1,-1)}*{\bullet}="A21" !{(-1.8,-1)}*{(1,4n-1)}
!{(-0.8,-1)}*[blue]{1}
!{(1,-1)}*{\bullet}="A22" !{(1.7,-1)}*{(2,4n-2)}
!{(0.7,-1)}*[blue]{3n}
% Level 2
!{(0,-2)}*{\bullet}="A32" !{(0.8,-2)}*{(0,4n-2)}
!{(0,-1.7)}*[blue]{2n}
!{(2,-2)}*{\bullet}="A33" !{(2.8,-2)}*{(1,4n-3)}
!{(1.8,-2.2)}*[blue]{3n^2}
% Level 3
!{(3,-3)}*{\bullet}="A44" !{(2.1,-3)}*{(0,4n-4)}  
!{(3.3,-3)}*[blue]{n^3}  
"A11": "A21" "A11": "A22"
"A21":"A32"  "A22":"A32"  "A22":"A33"  
"A33":"A44"
} 
\]
     \subcaption{$L_{v_1v_5}$}
     \end{subfigure}
       \hspace{.6 cm}
     \begin{subfigure}[b]{0.35\textwidth}
        \[ \xygraph{
% Level 0
!{(0,0)}*{\bullet}="A11" !{(0,0.2)}*{(2,4n)}
!{(0,-0.2)}*[blue]{1}
% Level 1
!{(-1,-1)}*{\bullet}="A21" !{(-1.5,-1)}*{(0,4n)}
!{(-0.8,-1)}*[blue]{1}
!{(1,-1)}*{\bullet}="A22" !{(1.7,-1)}*{(1,4n-1)}
!{(0.7,-1)}*[blue]{2n}
% Level 2
!{(2,-2)}*{\bullet}="A33" !{(2.8,-2)}*{(0,4n-2)}
!{(1.6,-2)}*[blue]{n^2}
"A11": "A21" "A11": "A22"
"A22":"A33"  
} 
\]
     \subcaption{$L_{v_1v_2}$}
     \end{subfigure}
     \caption{Recursive Lattice structure showing decomposition for $C_{4}\odot \bar K_n$}
        \label{fig:tree-structureC4corona}
\end{figure}

\begin{thrm}
 $C_{4}\odot \bar K_n$ is determined by its (signless) Laplacian permanental polynomial.     
\end{thrm}
\begin{proof}
Let $G$ be a graph having the same Laplacian permanental polynomial as $C_{4}\odot \bar K_n$. By applying Theorem \ref{C4corona-bark}, a detailed computation yields the following polynomial: 
\begin{align*}
     \psi(L(G); x) & =  x^{4n+4} -(8n+8)x^{4n+3} +(30n^2+58n+28) x^{4n+2} -\frac{212n^3+588n^2+520n+144}{3} x^{4n+1}\\ &+\cdots-(13n^2-4n-36)
\end{align*}
From this Laplacian polynomial, using the Lemma \ref{firstthreevalues}, we have 
\begin{align*}
    &\sum_{i=0}^4 k_i = 4n+4, 
    \sum_{i=0}^4 i k_i = 4n+4, 
    \sum_{i=0}^4 i^2 k_i = 4n^2+20n + 16,  &\\
    &-6t_G + \sum_{i=0}^4 i^3 k_i = 4n^3 + 24 n^2 + 52n+ 32. 
\end{align*}
A key observation follows from these equations: 
\begin{equation}
    6t_G - \sum_{i=0}^{d_{\max}} (i-1)(i-n-2)k_i = 0. \label{4eq:key}
\end{equation}

Solving this system uniquely determines the non-negative integer solution: $t_G = 0, k_0=k_2=\cdots=k_{n+1}=0,k_1=4n$, and $k_{n+2}=4$. This degree sequence describes a graph with $4n$ vertices of degree $1$, and four vertices of degree $n+2$. There is only one possible simple connected graph with this degree sequence, which is four vertices forming a cycle, and each having exactly $n$ many pendant vertices. Therefore, the graph is uniquely determined as $C_{4}\odot \bar K_n$. For the signless Laplacian case, it follows similarly.
\end{proof}

% \section{Conclusion}

\bibliographystyle{amsplain}

\end{document}